\documentclass[amscd,amssymb,verbatim,12pt]{amsart}
\usepackage{epsfig}
\usepackage{graphicx}
\newif\ifAMS
\IfFileExists{amssymb.sty}
  {\AMStrue\usepackage{amssymb}}
  {\usepackage{latexsym}}
  \usepackage{enumerate}
\theoremstyle{plain}

\newtheorem{theorem}{Theorem}[section]
     
\newtheorem{corollary}[theorem]{Corollary}

\newtheorem{question}{Question}
\newcommand {\Z}{{\mathbb Z}}
\newcommand {\N}{{\mathbb N}}
\newcommand {\R}{{\mathbb R}}

\title[Asymptotic topology of the lamplighter group]
 {Splittings and the asymptotic topology of the lamplighter group} 
\author{Panos Papasoglu}

\subjclass{20F65, 20E08 }

\email {papazoglou@maths.ox.ac.uk} 

\address
{Mathematical Institute, University of Oxford, 24-29 St Giles', Oxford, OX1 3LB, U.K.  }

\begin{document}
\maketitle

\begin{abstract}
It is known that splittings of finitely presented groups over
2-ended groups can be characterized geometrically. We show
that this characterization does not extend to all finitely
generated groups. Answering a question of Kleiner we show that
 the Cayley graph of the lamplighter group is coarsely
separated by quasi-circles.
\end{abstract}

\section{Introduction} 
\label{intro}

\noindent

Stallings \cite{St} showed that a finitely generated group has
more than one end if and only if it splits over a finite group.
This gives a geometric characterization of finitely generated
groups which split over a finite group, where geometric here means
in the sense of quasi-isometries.

One has a similar geometric characterization for splittings over
2-ended groups which applies to finitely presented groups. It was
shown in \cite{Pa} (see also \cite {Kl}) that a one ended finitely
presented group, which is not virtually a surface group, splits
over a 2-ended group if and only if its Cayley graph is coarsely
separated by a quasi-line. Another geometric characterization of
splittings of hyperbolic groups over 2-ended groups was given earlier by Bowditch \cite{Bo}. It is natural to ask whether the characterization given in
\cite{Pa} applies in fact to all one ended finitely generated
groups (as it is the case for Stallings' theorem).

We show here
that there is a one-ended finitely generated group (the lamplighter group)
which is coarsely separated by a quasi-line but does not split
over a 2-ended group. It turns out that the same group can be used to answer
a question of Kleiner (\cite {Pq}, problem 4.5).

\section{Preliminaries}

\subsection{Asymptotic Topology}
We recall some definitions relating to asymptotic topology.
The idea (due to Gromov) behind these definitions is that one can develop a large scale analog of topological notions. It is in a similar vein to large scale geometry except in asymptotic topology one tries to give a `rough' or `large scale' version of topological rather than geometric notions (see for example \cite{KK}, \cite{Dr}). Of course asymptotic topology notions are
`weaker' than asymptotic geometry notions, in particular they are invariant under quasi-isometries.

 A
(K,L)-{\it quasi-isometry} between two metric spaces $X,Y$ is a
map $f:X\to Y$ such that the following two properties are
satisfied:\newline 1) ${1\over K}d(x,y)-L\leq d(f(x),f(y))\leq
Kd(x,y)+L$ for all $x,y\in X$.\newline 2) For every $y\in Y$ there
is an $x\in X$ such that $d(y,f(x))\leq K$.\newline We will
usually simply say quasi-isometry instead of (K,L)-quasi-isometry.
Two metric spaces $X,Y$ are called quasi-isometric if there is a
quasi-isometry $f:X\to Y$.\newline A {\it geodesic metric space }
is a metric space in which any two points $x,y$ are joined by a
path of length $d(x,y)$. One turns a connected graph into a
geodesic metric space by giving each edge length 1.\newline
 We call a map $f:X\to Y$ between metric spaces a
{\it uniform embedding} (see \cite {Gr}) if the following two
conditions are satisfied:\newline 1) There are $K,L$ such that for
all $x,y\in X$ we have $d(f(x),f(y))<Kd(x,y)+L$.\newline 2) For
every $E>0$ there is $D>0$ such that $diam (A)<E\Rightarrow
diam(f^{-1}A)<D$.\newline If $f$
satisfies only condition 1 above we say that $f$ is a {\it coarse
lipschitz } map. We remark that a coarse lipschitz map is a uniform embedding if there is a function $h:\Bbb R ^+\to \Bbb R ^+$ such that for all
$A\subset Y$, $diam(f^{-1}A)\leq h(diam (A))$.
We say then that the {\it distortion } of $f$
is bounded by $h$. We remark that $f$ is a quasi-isometric embedding if its distortion is bounded by a linear function.\newline
It is easy to see that $f:X\to Y$ is a uniform embedding if for any two sequences $(a_n),\,(b_n)$ of $X$,
$d(a_n,b_n)\to \infty $ if and only if $d(f(a_n),f(b_n))\to \infty $.

 Although we will not need it in this paper we give one more characterization
of uniform embeddings which shows that it is a natural notion.
If $X$ is a metric space we say that $Y\subset X$ is large scale connected if there is an $R>0$, such that for any $x,y\in Y$ there is a finite sequence $x_0=x,x_1,...,x_n=y$ with $d(x_i,x_{i+1})\leq R$ for all $i$. If $Y$ is large scale
connected then for any sufficiently large $R$ one defines a metric $d_R$ on $Y$,
where $d_R(x,y)$ is the length of a shortest sequence $(x_i)$ joining $x,y$ and
such that $d(x_i,x_{i+1})\leq R$ for all $i$. Now the inclusion $Y\hookrightarrow X$ is a uniform embedding if for all $R,S$ sufficiently large the spaces $(Y,d_R), (Y,d_S)$ are quasi-isometric. 

If $K\subset X$ and $R>0$ we say that a component of $X-N_R(K)$ is
{\it deep } if it is not contained in $N_{R_1}(K)$ for any
$R_1>0$.\newline We say that $K$ {\it coarsely separates} $X$ if
there is an $R>0$ such that $X-N_R(K)$ has at least two deep
components. For example $X$ has more than one end if and only if $X$ is coarsely separated by a point.

Let $(X,d)$ be a geodesic metric space. If $R\subset X$ is path connected subset
of $X$, we denote the path metric of $R$ by $d_R$. We say that $R$ is a \textit{quasi-line} if $(R,d_R)$ is quasi-isometric to $\R$ and the inclusion
map $i:(R,d_R)\to (X,d)$ is a uniform embedding. We say that $R$ is a \textit{half-quasi-line} if $(R,d_R)$ is quasi-isometric to $\R ^+$.

Using this asymptotic notion of separation one may prove for example a generalization of Jordan's curve theorem for the plan (see e.g. \cite{KK}).

Kleiner has suggested the following asymptotic way to generalize the topological
notion of a separating simple closed curve. Let $X$ be a geodesic metric space.
Consider a sequence of simple closed curves $C_n$ in $X$ and denote by $d_n$ the path metric on $C_n$.  We say that $C_n$ is a \textit{sequence of quasi-circles}
if $diam\, C_n$ tends to infinity and there is a distortion function $h: \R ^+\to \R ^+$ such that $$d_n(a,b)\leq h(d(a,b))$$ for any $a,b$ lying on any $C_n$.

We say that a sequence of quasi-circles $C_n$ \textit{coarsely separates} $X$ if there is a $K$ (independent of $n$) such that $N_K(C_n)$ separates $X$ and at least two components of $X-N_K(C_n)$ are not contained in $N_n(C_n)$.

We may define similarly a sequence of quasi-intervals $I_n$ which coarsely separate $X$ by taking $I_n=[a_n,b_n]$ to be simple paths, rather than simple closed curves.

\subsection{Asymptotic topology and groups}
We remark that if $H$ is a finitely generated subgroup of a finitely generated group $G$ then the embedding $H\hookrightarrow G$ is a uniform
embedding (where we consider $H,G$ equipped with their word metrics).

It is easy to see that if a finitely generated group $G$ splits over a finitely generated group $H$ then $H$ coarsely separates the Cayley graph of $G$.
In the language of asymptotic topology Stallings ends theorem (\cite{St}) says that a finitely generated group $G$ splits over a finite group if and only if the Cayley graph of $G$ is coarsely separated by a point.

Clearly if a finitely generated group splits over a 2-ended group then a quasi-line coarsely separates the Cayley graph of $G$. It was shown in \cite{Pa}
that the converse also holds when $G$ is 1-ended, finitely presented and not
virtually a surface group. We will see in the next section that this characterization does not
generalize to finitely generated groups. 

Geometric characterizations of surface groups among finitely presented groups
or hyperbolic groups were used in \cite{Pa} and \cite{Bo} to show that splittings over 2-ended groups are invariant under quasi-isometries. Such
a characterization was a crucial ingredient also in the proof of the Seifert
conjecture by Mess \cite{Me}. It follows from \cite{Bos} that if a sequence
of quasi-circles coarsely separate the Cayley graph of a one-ended finitely presented
group $G$, then $G$ is virtually a surface group. We show in the
next section that the Cayley graph of the lamplighter group $L$ is coarsely
separated by quasi-circles, so the characterization above does not extend to
all one ended finitely generated groups.

\section{The asymptotic topology of the lamplighter group}

The lamplighter group $L$ is defined as
$$L=(\oplus _{ \Z}\Z _2)\rtimes \Z$$
where $\Z$ acts on the infinite direct sum $\oplus _{ \Z}\Z _2$
by sending the generator of the $i$-th factor, $x_i$ to the generator of
the next factor $x_{i+1}$ ($i\in \Z)$. It is easy to see that $L$ is in fact finitely generated and is given by the presentation
$$L=<a,t| a^2,[t^iat^{-i}, t^jat^{-j}],\, i,j\in \Z >$$
We will denote by $X$ the Cayley graph of $L$ with respect to the generators $a,t$.

\subsection{Elements of the lamplighter group and metric properties}
It is convenient to represent elements of $L$ geometrically: We
consider bi-infinite words of $0,1$'s indexed by $\Z$ and a
cursor. The identity element is given by a bi-infinite string of
$0$'s with the cursor at the position $0$ of $\Z$. Any other
element is represented by a string containing at most finitely
many $1$'s with the cursor at some position $k$.

One thinks of these strings of $0,1$'s as lamps at the integers on the real line. $0$ indicates a lamp which is off, while an $1$ indicates a lamp which is lit. The cursor is the lamplighter. If we think that it takes 1 unit of time for
the lamplighter to move from the position $n$ to the position $n+1$ (or $n-1$) and 1 unit of time to turn the switch on or off, then the distance betweeen the identity configuration (lamps off, cursor at $0$) to a given configuration is given by the time it takes the lamplighter to turn on the lamps and to move
to the position of the given configuration. For a more detailed exposition
of normal forms and the geometric representation of the elements of $L$ we refer to \cite{CT}.

\subsection{A separating quasi-line in the lamplighter group}
We show now that there is a half-quasi-line which coarsely separates the
Cayley graph of $L$ (problem 4.2 in \cite{Pq}). This contrasts with the case of finitely presented groups
where it is impossible for a half-quasi-line to coarsely separate the Cayley graph of a 1-ended group. 

\begin{theorem} \label{halfq}
There is a half-quasi-line $N$ which separates coarsely the Cayley graph of $L$, $X$.
\end{theorem}
\begin{proof}
We describe $N$ using the geometric representation of the elements of $L$ described earlier. We split the proof in 3 parts. In the first part we define $N$. In the second part we show that $N$ is uniformly embedded. In the third part we show that it separates coarsely $X$.

(1) $N$ will be defined as a graph isomorphic to half-line given by a sequence of vertices of $X$. Using our geometric representation we may define $N$ as an infinite `trip' of the cursor (`lamplighter') which moves on the line, turning switches on and off.

We associate to each finite configuration $c$ of $0,1$'s on $\Z$ a dyadic integer $c^+$, where $c^+$ is the dyadic number we consider only the digits of $c$ which are right from $0$ or at $0$. Similarly we associate to $c$ the dyadic number $c^-$ which is the number which we read from $c$ if we start at $-1$ and move left. So for example if $c$ is the configuration:

$$c=...0,0,1,0,S0,1,1,0,0,...$$
where we denote by $S$ the origin, then $c^+=011=6$ and $c^-=10=2$.
If $c$ is a configuration we denote by $c(k)$ the digit at the position $k$.
So for instance in the example above we have $c(0)=0,c(-2)=1,c(1)=2$.

The first vertex of $N$ is the identity configuration: the lamplighter at $0$
and all lamps at the $0$ state (off). We describe now how the lamplighter moves.
At stage $n$ ($n\in \N$) the lamplighter is at the configuration $c_n$ where the lamplighter is at the origin, $c_n^+$ is the number $n$ written in dyadic and there are only $0$'s left of the lamplighter. Now to go from the stage $n$ to the stage $n+1$ the lamplighter does the following: If $c_n(i)=1$ for all $i$,
$0\leq i\leq k-1$ and $c_n(k)=0$ then the lamplighter moves left to $-k$, switches it to $1$ and moves back to the origin turning the switch at position
$-k+r$ so that it is the same as the switch at position $k+r$ (for all $r\leq k-1$). We denote the configuration abtained after this by
 $c_{n0}$. After that the lamplighter moves to the right switching as he moves all lamps that are on to off till he reaches the $k$-th lamp which is off and switches it on. If we denote by $c_{n1}$ the configuration that results from this, it is easy to see that $$c_{n1}^+=n+1$$
After that the lamplighter moves again left till he reaches the origin, after the origin he goes left turning off all lamps which are on, and after turning off the last lamp which is on, at position $-k$, returns to the origin. In this way the
lamplighter moves from $c_n$ to $c_{n+1}$. It is clear that the lamplighter does not pass twice from the same configuration, so $N$ is indeed a half line.

(2) We show now that $N$ is uniformly embedded. Let's denote by $N(i)$ the $i$-th vertex of $N$. It is enough to show that there
are no sequences $a_k=N(i_k),b_k=N(j_k)$ such that
$$|i_k-j_k|\to \infty $$ and
$$d(a_k,b_k)=M $$
for some fixed $M\in \N$.

It is enough to show the following: If $a=N(i),\, b=N(j)\, (i<j)$ and $d(a,b)\leq M$ then $j-i$ is bounded by a constant $D=D(M)$ which depends on $M$.
Let's denote by $l(a),l(b)$ the cursor positions for $a,b$ respectively.
We distinguish now some cases:

Case 1: $l(a),l(b)\in [-M,M]$. Then the configurations of $a,b$ are identical in $[2M,\infty ]$. This implies that when the we move from $N(i)$ to $N(j)$ the cursor moves in the interval $[-2M,2M]$. But then there is a bounded number of
new configurations that can be created so $j-i$ is bounded in this case by a constant $D_1=D_1(M)$.

Case 2: At least one of $l(a),l(b)$ lies in $[M+1,\infty )$. Then both $l(a),l(b)$ lie in $[1,\infty )$. It is clear that if there is some $n$ so that both $a,b$ are configurations that appear as we move from $c_n$ to $c_{n+1}$ then $j-i\leq M$. We assume therefore that $a$ is a configuaration that appears in the stage
$c_{k}$ to $c_{k+1}$, while $b$ is a configuration that appears in the stage $c_n$ to $c_{n+1}$ for some $n>k$. Clearly $|l(a)-l(b)|\leq M$. We recall now that as
the cursor moves from $c_{k0}$ to $c_{k+1}$ it erases a string of consecutive $1$'s. Let's say that the initial string of $1$'s of $c_k$ consists of $d_k$ $1$'s and the initial string of $1$'s of $c_n$ consists of $d_n$ $1$'s. If $d_k\ne d_n$ then $d(a,b)>M$. This is because, by the way they were
defined $c_{k0}^-$ and $c_{n0}^-$ differ at the position $min (-d_k,-d_n)$ (and it takes the cursor more than $M$ steps to change this, since one cursor is in
$[M+1,\infty )$).

We assume therefore that $d_k=d_n=d>M$. Since $d(a,b)=M$ we have that $c_{k0}^-=c_{n0}^-$. This however implies that the configurations $c_k$, $c_n$ are identical in all positions in $[0,2d]$. Since $d(a,b)=M$  we have also that
$c_k,c_n$ are identical in $[d+M,\infty )$. But $d>M$, so $c_n=c_k$, which is
a contradiction since we assumed that $n>k$.

Case 3: At least one of  $l(a),l(b)$ lies in $(\infty , -M-1]$.
As in case 2 we see that $a$ is a configuaration that appears in the stage
$c_k$ to $c_{k+1}$, while $b$ is a configuration that appears in the stage $c_n$ to $c_{n+1}$ for some $n>k$. However since $k\ne n$ we have that $c_k^+\ne c_n^+$. On the other hand one of $l(a),l(b)$ lies in $(\infty , -M-1]$. This implies that $d(a,b)>M$.

(3) We show finally that $N$ separates coarsely $X$. To show this it is enough to show that there are sequences of vertices $a_n,b_n$ in $X$ such that
$$d(a_n,N)\geq n,\,\, d(b_n,N)\geq n$$ and $a_n,b_n$ lie in distinct components of $X-N$.  We define $a_n$ to be the following configuration: $a_n^+$ consists of $2n$ consecutive $1$'s, $a_n^-$ consists of $0$'s and the cursor is at the
position $n$. $b_n$ is the configuration with $0$'s everywhere and the cursor at $-n$. It is clear now that to move from $a_n$ to $N$ the cursor has to erase
all $1$'s at the interval $[0,n]$, so $d(a_n,N)\geq n$. Similarly to move from $b_n$ to $N$ the cursor has to move from position $-n$ to $0$ so $d(b_n,N)\geq n$.

We show now that $a_n,b_n$ are separated by $N$. To move from $b_n$ to $a_n$ the cursor has to go through $0$. But no matter how the cursor switches lamps on and off before reaching $0$, when the cursor moves to $0$ we are exactly at the begining of the stage $c_k$, where $k$ is the integer whose dyadic expression we see on the right side. So any path joining $a_n,b_n$ in $X$ intersects $N$.

\end{proof}

It is not very hard to see that $X$ is separated also by a quasi-line. Indeed we define a quasi-line $R$ by $R(i)=N(i)$ for all $i\in \N$ and we define $R(-i)$ to be the cursor at $-i$ and the lamps in $[-i,0]$ are on (and all other lamps off). One sees easily that $R$ is uniformly embedded in $X$.
It is clear that if we take $a_n,b_n$ as in the previous proof we have that $d(a_n,R)\geq n,\, d(b_n,R)\geq n$ so $R$ coarsely separates $X$. That is we have the following:
\begin{corollary}\label{quasil}
There is a quasi-line $R$ which separates coarsely the Cayley graph
of $L$, $X$.
\end{corollary}
It is easy to see that $L$ is one-ended and does not split over a 2-ended group. Indeed assume that $L$ acts on a tree $T$. Then if we see $L$ as a semi-direct product
$$L=(\oplus _{ \Z}\Z _2)\rtimes \Z$$ the direct sum $\oplus _{ \Z}\Z _2$ fixes some vertex $v$ of the tree. If the generator $t$ of $\Z$ sends $v$ to $t(v)\ne v$ then since $tx_it^{-1}=x_{i+1}$ (where $x_i$ is the generator of the $i$-th $\Z _2$ factor) we have that $\oplus _{ \Z}\Z _2$ fixes also $t(v)$, so there is an edge fixed by $\oplus _{ \Z}\Z _2$. Clearly if $t(v)=v$ then $L$ fixes a vertex of the tree. In both cases we see that there is no action with 2-ended edge stabilizers.
\subsection {Quasi-circles}
We show now that there is a sequence of quasi-circles that separate coarsely the Cayley graph of $L$, $X$. In fact we show something slightly stronger: a sequence of
quasi-intervals coarsely separate $X$. The construction of the quasi-circles (or quasi-intervals) is similar to the construction of $N$.

We fix $n\in \N$. We define now a simple path $I_n$ and a simple closed path $C_n$ such that
there are $a_n,b_n\in X$ which are separated by $I_n$ ($C_n$) and such that
$$d(a_n,I_n)>n,\, d(b_n,I_n)>n,\, d(a_n,C_n)>n,\, d(b_n,C_n)>n$$
The interval $I_n$ is defined as follows: We think of $I_n$ as a
path $I_n(i)$, $i\in \N$. If $N$ is the half-quasiline defined in the previous section we set $I_n(i)=N(i)$ for all $i$ till we reach the configuration
consisting of $1$'s in the positions $0,1,...,2n$, the cursor at the origin and $0$ everywhere else. After this the lamplighter moves
from position $0$ to position $2n$ switching off the lamps as he goes. Let's say that $I_n(k)$ is the vertex of $X$ that we get when the
lamplighter reaches $2n$. 

When the lamplighter reaches $2n$ he repeats exactly the
steps that he did at the beginning except he walks in the inverse
direction. More formally, let's denote by $q(i)$ the walk of the
lamplighter, where for any $i$, $q(i)$ is either a move by one to
the right denoted by $+1$ or a move to the left, denoted by $-1$
or a change of switch denoted by $0$. Then if $q(k)$ is the
lamplighter at $I_n(k)$ we define $q(k+i)=-q(i)$. The lamplighter
continues till he reaches the following configuration: lamplighter
at $2n$ and configuration consisting of $1$'s in the positions
$0,1,...,2n$, and $0$ everywhere else.
We will show now that the sequence $I_n$ is a sequence of quasi-intervals.
In order to analyse distortion we think of $I_n$ as constructed in 3 stages.
So we write $$I_n=I_{n1}\cup I_{n2}\cup I_{n3}$$ where
$I_{n1}, I_{n2}, I_{n3}$ are defined as follows: $I_{n1}$ is the initial path,
so $I_{n1}=I_n\cap N$. $I_{n3}$ is the last stage where the lamplighter moves
exactly as in $I_{n1}$ but in the inverse direction and starting from $2n$ rather than $0$. Finally $I_{n2}$ is the path of $I_n$ joining $I_{n1},I_{n3}$.
In order to show that $I_n$ is a sequence of quasi-intervals it is enough
to show that there are no sequences $a_n,b_n\in I_n$ such that
$$a_n=I_n(i_n),\, b_n=I_n(j_n)$$
with $$\limsup _{n\to \infty} |i_n-j_n|=\infty $$
and $d(a_n,b_n)=M$ for some fixed $M\in \N$.
Since $N$ is uniformly embedded in $X$ it is clear that it is not possible that $a_n,b_n$ are contained both in $I_{n1}$ or in $I_{n3}$. Since $I_{n2}$ is a geodesic path it follows that $a_n,b_n$ can not be contained both in $I_{n2}$ either.
Finally since $d(I_{n1},I_{n3})>n$ it is not possible that one of the two is contained in $I_{n1}$ and the other in $I_{n3}$. So there are 2 cases left:

Case 1: $a_n\in I_{n1},\, b_n\in I_{n2}$.

Case 2: $a_n\in I_{n2},\, b_n\in I_{n3}$.

Assume now that we are in case 1 and consider the geometric
representation of $a_n,b_n$. If the cursor for $b_n$ is at $[0,M]$ then,
since $d(a_n,b_n)\leq M$ the configurations for $a_n,b_n$ are identical in
$[2M,\infty )$. But, by the definition of $I_n$, this implies that $j_n-i_n$ is bounded by some constant $D(M)$ depending only on $M$. If the cursor for $b_n$ lies in $[M+1,2n)$ we remark that either the cursor for $a_n$ is at $0$, so
$d(a_n,b_n)>M$, or the configuration of $a_n$ has some lamp on in $[-2n,1]$. However in this case too $d(a_n,b_n)>M$.

We assume now that we are in case 2. If the cursor for $a_n$ lies in
$[0,2n-M-1]$ then clearly $d(a_n,b_n)>M$. If the cursor for $a_n$ lies in
$[2n-M,2n]$ then the configuration of $b_n$ has to be identical with the configuration of $a_n$ in $[0,2n-M]$ (since $d(a_n,b_n)=M$). However this implies that $j_n-i_n$ is smaller than a constant $D(M)$ that depends only on $M$. We see in both cases that it is impossible to have
$$\limsup _{n\to \infty} |i_n-j_n|=\infty $$
so $I_n$ is indeed a sequence of quasi-intervals.

We show now that there are $x_n,y_n$ in $X$ such that $$d(x_n,I_n)\geq n,\,d(y_n,I_n)\geq n$$ and $x_n,y_n$ lie in distinct components of $X-I_n$.
We take $x_n$ to be the following configuration: cursor at $n$, $1$'s at the interval $[0,2n]$ and $0$'s everywhere else. We take $y_n$ to be the cursor
at $-n$ and $0$ everywhere. By the way they are defined $x_n,y_n$ satisfy
$d(x_n,I_n)\geq n,\,d(y_n,I_n)\geq n$. It is clear that to move from
$x_n$ to $y_n$ one has to go through some vertex in $I_{n1}\cup I_{n3}$.
This shows that the sequence of quasi-intervals $I_n$ coarsely separates $X$.

It is easy to define a sequence of quasi-circles $C_n$ which coarsely separates $X$, using the quasi-intervals $I_n$. Namely we join the endpoint of
$I_n$ to the initial point as follows: The lamplighter moves from $2n$ back to $0$ turning off as he goes all lamps in the interval $[0,2n]$. One sees as before that the $C_n$'s are uniformly embedded (i.e. it is a sequence of quasi-circles). If we pick $x_n,y_n$ as before we see that 
$$d(x_n,C_n)\geq n,\,d(y_n,C_n)\geq n$$ and $x_n,y_n$ are separated by $C_n$ since they are already separated by $I_n\subset C_n$.

\section{Discussion}
The idea behind the coarse separation properties of the lamplighter group
comes from the fact that the asymptotic dimension of the lamplighter group is 1
(\cite{Ge}). So one can think of the lamplighter group as a group theoretic (or `asymptotic') analog of the Menger curve. It is easy to see that the Menger curve is separated by simple paths (which is not very surprising in itself, since the topological dimension of the Menger curve is 1 it is clearly separated by a totally disconnected set). This remark prompts us to reformulate
some of the questions treated in this paper.

\begin{question} 
Let $G$ be a one-ended finitely generated group and let $X$ be the Cayley graph of $G$. Assume that no subset of asymptotic dimension $0$ of $X$ coarsely separates $X$.
Is it true that no sequence of quasi-circles coarsely separates $X$? Is it true that no half-quasi-line coarsely separates $X$?
\end{question}

\end{document}